\documentclass[twoside,12pt]{amsart}
\usepackage{amssymb}
\usepackage{mathtools}
\usepackage{enumitem}
\usepackage{tikz}
\usepackage{bbm, color}
\numberwithin{equation}{section}

\newtheorem{Theorem}{Theorem}[section]
\newtheorem{Lemma}[Theorem]{Lemma}

\theoremstyle{definition}

\theoremstyle{remark}

\def\eps{\epsilon}

\def\0{{\bar 0}}
\def\1{{\bar 1}}

\def\C{\mathbb C}
\def\Z{\mathbb Z}

\def\g{{\mathfrak g}}
\def\gl{\mathfrak{gl}}

\def\so{\mathfrak{so}}
\def\sp{\mathfrak{sp}}

\def\col{\operatorname{col}}

\def\End{\operatorname{End}}
\def\Hom{\operatorname{Hom}}

\def\row{\operatorname{row}}

\def\mod-{\text{mod-}}
\def\-mod{\text{-mod}}

\newdimen\Hoogte    \Hoogte=12pt    % hoogte  van hokje
\newdimen\Breedte   \Breedte=12pt   % breedte van hokje
\newdimen\Dikte     \Dikte=0.5pt    % dikte lijn

\newenvironment{Young}{\begingroup
       \def\vr{\vrule height0.8\Hoogte width\Dikte depth 0.2\Hoogte}
       \def\fbox##1{\vbox{\offinterlineskip
                    \hrule height\Dikte
                    \hbox to \Breedte{\vr\hfill##1\hfill\vr}
                    \hrule height\Dikte}}
       \vbox\bgroup \offinterlineskip \tabskip=-\Dikte \lineskip=-\Dikte
            \halign\bgroup &\fbox{##\unskip}\unskip  \crcr }
       {\egroup\egroup\endgroup}

\def\Diagram#1{\relax\ifmmode\vcenter{\,\begin{Young}#1\end{Young}\,}\else%
              $\vcenter{\,\begin{Young}#1\end{Young}\,}$\fi}

\title[]
{\boldmath Finite $W$-algebra invariants via Lax type operators}
\author{Jonathan Brown}
\address{ Department of Mathematics, Computer Science and Statistics,
                      State University of New York,
                      Oneonta, NY 13820, USA}
\email{Jonathan.Brown@oneonta.edu}

\begin{document}

\begin{abstract}
We use variations on Lax type operators to find explicit formulas for certain elements of finite $W$-algebras.
These give a complete set of generators for all finite $W$-algebras of types
B,C,D for which the Dynkin grading is even.

\end{abstract}

\maketitle

\section{Introduction}
In \cite{SKV} the authors construct a Lax type operator using Dirac reductions and generalized  quasideterminants
to find an explicit formula for generators of a subalgebra of a finite $W$-algebra.
Their formula applies to all types, and it includes the highest degree generators (as defined in $\S$5) and thus most computationally difficult to find
generators of the finite $W$-algebra.
In this paper we develop operators
based on the Lax type operator in \cite{SKV}.  These operators produce explicit formulas for additional finite $W$-algebra elements beyond those found in \cite{SKV}, and in some cases gives generators for the finite $W$-algebra.

We use the definition of finite $W$-algebras from \cite[$\S$2.1]{BGK}, except following \cite{SKV} we swap the roles of $e$ and $f$.
For this definition,
let $\g$ be a reductive Lie algebra over $\C$ equipped with
a symmetric non-generate equivariant
bilinear form 
$\langle \cdot \mid  \cdot \rangle$, and let $f$ be a nilpotent element in $\g$.
To define the finite $W$-algebra,
embed $f$ into an $\mathfrak{sl}_2$-triple $(f, 2x, e)$.
For $k \in \mathbb{Z}$ let
$\g_k \subset \g$
denote the set of eigenvectors for $\operatorname{ad} x$ with eigenvalue $k$.
Let $J$ be the left $U(\g)$ ideal generated by $\{m - \langle f \mid m \rangle  \mid m \in \g_{\geq 1} \}$. 
Now the finite $W$-algebra is defined via
\begin{equation} \label{L:walgdef}
U(\g,f) = (U(\g)/J)^{\g_{\geq 1/2}}.
\end{equation}

The Lax type operator depends on $\phi : \g \to \operatorname{End}(V)$,
a faithful finite-dimensional representation of $\g$.
For $k \in \mathbb{Z}$ let
$V[k] \subset V$
denote the set of eigenvectors for $x$ with eigenvalue $k$.
(We will also define the subspaces $V[<k], V[\geq k]$, etc in the obvious way.)
In general if we talk about the weight of a vector with no other context, we means its eigenvalue for $x$.
Let
$\tfrac{\xi}{2}$ be the highest eigenvalue for $x$ on $V$.
Let $\{u_i\}_{i \in I}$ be a basis of $\g$ consisting  $\operatorname{ad} x$ eigenvectors,
let $\{u^i\}_{i \in I}$ denote the dual basis with respect to the form on
$\g$, let $U_i = \phi(u_i)$, let $U^i = \phi(u^i)$, and let $F = \phi(f)$.
For a homogeneous $y \in \g_k$ or $v \in V[k]$ let
\[
\delta(y) =\delta(v) = k.
\]
Let
\begin{equation} \label{EQ:D}
  D =  \sum_{u_i \in \g_{\geq 1}} U^i U_i \in \End(V)
\end{equation}
be a diagonal ``shift'' matrix.
Define
\[
Y = \mathbbm{1}_V + F+\sum_{u_i \in \g_{\leq 1/2}} z^{\delta(u_i)-1} u_i \otimes U^i + z^{-1} D
\in U(\g)[[z^{-1/2}]]\otimes \End(V).
\]
Here and throughout this paper we use the notation $\sum_{u_i \in \g}$ to denote the sum over our basis elements $\{u_i\}$.

For $k \in \tfrac{1}{2} \mathbb{Z}$
let $\Psi_{-k} : V[-k] \hookrightarrow V$ denote inclusion and $\Pi_{k} : V \twoheadrightarrow V[k]$ be projection.
Define $\Psi_{< d /2}, \Phi_{> -\tfrac{\xi}{2}}$, $\Psi_{[k,k)}$ etc similarly.
Also if $W$ is a subspace of $V$ spanned by certain weight vectors,
let the complement of $W$ be the span
of all the other weight vectors.
Now $\Pi_W$ and $\Psi_W$ denote projection and inclusion for $W$.
In \cite{SKV} they use the Lax type operator to generator certain elements of $U(\g,f)$.
This operator is defined via
\[
 L(z) = 
 \Pi_{\tfrac{\xi}{2}} Y \Psi_{-\tfrac{\xi}{2}}  - \Pi_{\tfrac{\xi}{2}} Y \Psi_{> -\tfrac{\xi}{2}} \left(\Pi_{<\tfrac{\xi}{2}} Y \Psi_{>-\tfrac{\xi}{2}}\right)^{-1} \Pi_{< \tfrac{\xi}{2}} Y \Psi_{-\tfrac{\xi}{2}},
\]
which is an element of 
$U(\g)[[z^{-1}]] \otimes \operatorname{Hom}(V[-\xi/2],V[\xi/2])$.
Actually this is a power of $z$ times the Lax type operator studied in \cite{SKV}; see $\S2$ below.
In this paper we 
use the Lax type operator and a new related operator to find formulas for additional elements of $U(\g,f)$
mainly in the case that $\g$ is of classical type.
In many cases these elements generate $U(\g,f)$.

The new operator we call the {\em right handed Lax type operator}.  It is defined via
\begin{align} \label{EQ:LR}
L_{R}(z) &= \left(\Pi_{<\tfrac{\xi}{2}} Y \Psi_{>-\tfrac{\xi}{2}}\right)^{-1} \Pi_{<\tfrac{\xi}{2}} Y \Psi_{-\tfrac{\xi}{2}} \\ \notag
& \quad \quad \in U(\g)[[z^{-1}]] \otimes \operatorname{Hom}(V[-\xi/2],V[>-\xi/2]).
\end{align}

Note that $L_{R}(z)$ occurs as a subexpression of $L(z)$.
Our first main theorem states that some of coefficients of $L_{R}(z) \bar 1$ are in $U(\g,f)$,
where $\bar 1$ denotes the image of $1$ in $U(\g)/J$.
\begin{Theorem} \label{T:main}
  Suppose that $V$ has highest weight vectors for the $\mathfrak{sl}_2$-triple $(f,2x,e)$ of different weights.
  Let $l \in \frac{1}{2} \Z$ be the largest number such that that $l < \xi/2$ and $l$ is the weight of a highest weight vector in $V$.
  Let $b \in \frac{1}{2}\Z$ be such that 
  $b > \xi/2-l$.
  Let $W$ be the subspace of V spanned by lowest weight vectors of weight $-l$.
  Then the coefficients of $z^{-b}$ in $\Pi_W L_{R}(z) \Psi_{-\tfrac{\xi}{2}} \bar 1$
  are  in $U(\g,f)$.
\end{Theorem}

Our remaining results only apply when $\g$ is of classical type
and where all the parts of the Jordan type of $f$
have the same parity.

For a non-negative integer $k$
let
\begin{equation} \label{EQ:gk}
\g^k = \{ \Pi_{[-k,k]} X \Psi_{[-k,k]} \mid X \in \g \}
\end{equation}
where we are identifying $\g$ as a
subset of $\operatorname{End}(V)$.
Be careful not to confuse $\g^k$ with $\g_k$ defined above.
It is easy to see that
\[
  \left[ \Pi_{[-k,k]} X \Psi_{[-k,k]},
   \Pi_{[-k,k]} Y \Psi_{[-k,k]}  \right]
=   \Pi_{[-k,k]} [X,Y] \Psi_{[-k,k]},
\]
so $\g^k$ is a Lie subalgebra of $\g$.
Also let
\begin{equation} \label{EQ:fk}
f_k = \Pi_{[-k,k]} f \Psi_{[-k,k]}
\end{equation}
so that $f_k$ is a nilpotent element of $\g^k$.

Let $L_k(z)$ be the operator $L(z)$ defined for the Lie subalgebra
$\g^k$ and $f_k$.  Some caution is required here:  If $D_k$ is the
shift matrix used to calculate $L_k(z)$ from \eqref{EQ:D}, then
$D_k \neq \Pi_{[-k,k]} D \Psi_{[-k,k]}$.

\begin{Theorem} \label{T:Lk}
  Suppose that $\g$ is a Lie algebra of
  classical type and that $f \in \g$ is a nilpotent element with Jordan type with all parts having the same parity.
  Let $k \in \tfrac{1}{2} \Z_{\geq 0}$.
  Considering $V$ to a module of the $\mathfrak{sl}_2$-triple $(f,2x,e)$
 let $U$ be the subspace of $V$ spanned by the $\mathfrak{sl}_2$- submodules of $V$ of highest weight $k$.
 Then
     the coefficients of
     $\Pi_U L_k(z) \Psi_U \bar 1$ lie in $U(\g,f)$.
 \end{Theorem}

Let $L_{k;R}(z)$ be the operator $L_R(z)$ from \eqref{EQ:LR}
defined for $\g^k$ and $f_k$.
 \begin{Theorem} \label{T:LRk}
  Let $\g, f, k$, and $U$ be as in the previous theorem.
  Let $m \in \frac{1}{2} \Z$ be the largest number such that that $m < k$ and $m$ is the weight of a highest weight vector in $V$,
  Let $b \in \frac{1}{2}\Z$ be such that
  $b > k-m$.
  Let $W$ be the subspace spanned by lowest weight vectors of weight $-m$ in $V$.
  Then the coefficients of $z^{-b}$ in
  $\Pi_W L_{k;R}(z)\Psi_U \bar 1$
  are  in $U(\g,f)$.
 \end{Theorem}

\begin{Theorem} \label{T:generate}
   Then  the union of the elements of $U(\g,f)$ obtained from
   the previous two theorems for the values of $k$ which are highest weights for $V$ for the $\mathfrak{sl}_2$-triple $(f,2x,e)$ generate $U(\g,f)$.
\end{Theorem}

The ultimate goal is to obtain presentations for all finite $W$-algebras.  In the future work \cite{B2} we will use these generators to obtain presentations of two-row finite $W$-algebras associated to classical Lie algebras, where by two-row we mean that the Jordan type of $f$ has exactly two parts.

\subsection*{Acknowledgements}
We would like to thank Simon Goodwin for his hospitality and many enlightening conversations.

\section{The Lax type operator}

The Lax type operator studied in \cite{SKV} is defined as the quasideterminant
\begin{align*}
\hat L(z) = & \left(  \Pi_{-\tfrac{\xi}{2}} \left( z \otimes \mathbbm{1}_V + F + \sum_{u_i \in \g_{\leq 1/2}}  u_i \otimes U^i  +  D\right)^{-1} \Psi_{\tfrac{\xi}{2}} \right)^{-1}  \\
 & \qquad \qquad \in U(\g)((z^{-1})) \otimes \operatorname{Hom}(V[-\xi/2],V[\xi/2]).
\end{align*}
Now \cite[Theorem 4.9]{SKV} says that 
the coefficients of $\hat L(z) \bar 1$
are in $U(\g,f)$,
where $\bar 1$ is the image of 1 in $U(\g)/J$. 
However these coefficients only generate $U(\g,f)$ in the case
that all of the highest weight vectors in $V$ for the $\mathfrak{sl}_2$-triple $(e,2x,f)$ have the same weight.
For classical Lie algebras, this translates to $f$ being
is a rectangular nilpotent element, that is all the parts of the Jordan form of $f$ have the same size.

By \cite[(5.25)]{SKV} $\hat L(z)$ can be expressed (using a Dirac reduction) as
\begin{align*}
z^{-1-d} & \hat L(z)  = \\
  & \Pi_{\tfrac{\xi}{2}} Y \Psi_{-\tfrac{\xi}{2}}  - \Pi_{\tfrac{\xi}{2}} Y \Psi_{> -\tfrac{\xi}{2}} \left(\Pi_{< \tfrac{\xi}{2}} Y \Psi_{> -\tfrac{\xi}{2}}\right)^{-1} \Pi_{< \tfrac{\xi}{2}} Y \Psi_{-\tfrac{\xi}{2}}.
\end{align*}
Thus $z^{-d-1} \hat L(z) = L(z)$.

Some results from \cite{SKV} make calculating commutators with $L_r(z)$ tractable.
Let 
\begin{equation} \label{EQ:Z}
Z = \mathbbm{1}_V + \sum_{u_i \in \g} z^{\delta(u_i)-1} u_i \otimes U^i.
\end{equation}
It turns out key calculations involving $Y$ can be done using $Z$, for which is is easier to calculate commutators.
More precisely,
 \cite[(5.26)]{SKV}
 says that
\begin{equation} \label{EQ:ZY}
 (\Pi_{<\tfrac{\xi}{2}} Z \Psi_{>-\tfrac{\xi}{2}})^{-1} \Pi_{<\tfrac{\xi}{2}} Z \Psi_{-\tfrac{\xi}{2}} \bar 1=
 (\Pi_{<\tfrac{\xi}{2}} Y \Psi_{>-\tfrac{\xi}{2}})^{-1} \Pi_{<\tfrac{\xi}{2}} Y \Psi_{-\tfrac{\xi}{2}} \bar 1.
\end{equation}
 Thus it is enough to calculate certain commutators
 in the above expression involving $Z$ instead of $Y$.

The following lemma is essentially \cite[(5.49)]{SKV}:
\begin{Lemma} \label{L:SKV1}
Let $m \in \g_k$ for some $k$.  Then
\[
 [m, Z] = z^{-\delta(m)} [Z, \phi(m)].
\]
\end{Lemma}

Now we use these results
to prove the following lemma.
\begin{Lemma} \label{L:help}
Let $m \in \g_{\geq 1/2}$.
Then
\begin{align*}
    &\left[m,
\left(\Pi_{<\tfrac{\xi}{2}} Z \Psi_{>-\tfrac{\xi}{2}}\right)^{-1} \Pi_{<\tfrac{\xi}{2}} Z \Psi_{-\tfrac{\xi}{2}}
    \right]
    \bar 1
    =
     z^{-\delta(m)} \left( \Pi_{>-\tfrac{\xi}{2}} \phi(m) \Psi_{-\tfrac{\xi}{2}}\right. \\
    &\qquad \qquad - \left.
       \Pi_{>-\tfrac{\xi}{2}} \phi(m) \Psi_{>-\tfrac{\xi}{2}}
\left(\Pi_{<\tfrac{\xi}{2}} Z \Psi_{>-\tfrac{\xi}{2}}\right)^{-1} \Pi_{<\tfrac{\xi}{2}} Z \Psi_{-\tfrac{\xi}{2}}
       \right) \bar 1.
       \end{align*}
\end{Lemma}
\begin{proof}
   To use space more efficiently, let
   $\pi_{\xi/2}(A) = \Pi_{<\tfrac{\xi}{2}} A \Psi_{>-\tfrac{\xi}{2}}$ and
   $\pi_{\xi/2}^R(A) = \Pi_{<\tfrac{\xi}{2}} A \Psi_{-\tfrac{\xi}{2}}$.
   We calculate
   \begin{align*}
   z ^{\delta(m)} &
   \left[m, \pi_{\xi/2}(Z)^{-1} \pi_{\xi/2}^R(Z) \right]  \\
   & =
   z ^{\delta(m)}
   \left[m, \pi_{\xi/2}(Z)^{-1}\right] \pi_{\xi/2}^R(Z)
    +
    \pi_{\xi/2}(Z)^{-1}
    \left[m, \pi_{\xi/2}^R{Z}\right]   \\
   & = 
   z ^{\delta(m)}  (
   -    
    \pi_{\xi/2}(Z)^{-1}
   \left[m, \pi_{\xi/2}(Z) \right]
    \pi_{\xi/2}(Z)^{-1}
   \pi_{\xi/2}^R(z)  \\
    & \qquad +
    \pi_{\xi/2}(Z)^{-1}
    \left[m, \pi_{\xi/2}^R(Z) \right] )\\
   & =
   z ^{\delta(m)}  (
   -    
    \pi_{\xi/2}(Z)^{-1}
   \pi_{\xi/2}(\left[m,Z\right])
    \pi_{\xi/2}(Z)^{-1}
    \pi_{\xi/2}^R(Z)\\
    & \qquad +
    \pi_{\xi/2}(Z)^{-1}
     \pi_{\xi/2}^R([m,Z])).
   \end{align*}
   By Lemma \ref{L:SKV1} this equals
   \begin{align} \label{EQ:line0}
   & -    
    \pi_{\xi/2}(Z)^{-1}
   \pi_{\xi/2}([Z,\phi(m)])
     \pi_{\xi/2}(Z)^{-1}
    \pi_{\xi/2}^R(Z)   \\
    & \notag  \qquad  +
    \pi_{\xi/2}(Z)^{-1}
     \pi_{\xi/2}^R([Z,\phi(m)])   \\
    & = 
    \label{EQ:line1}
    \pi_{\xi/2}(Z)^{-1}
   \pi_{\xi/2}(\phi(m) Z)
    \pi_{\xi/2}(Z)^{-1}
     \pi_{\xi/2}^R(Z)  \\
  \label{EQ:line2} 
     &\qquad - \pi_{\xi/2}(Z)^{-1}
   \pi_{\xi/2}(Z \phi(m))
    \pi_{\xi/2}(Z)^{-1}
    \pi_{\xi/2}^R(Z) \\
   \label{EQ:line3}
     & \qquad +
    \pi_{\xi/2}(Z)^{-1}
     \pi_{\xi/2}^R(Z \phi(m))  \\
     \label{EQ:line4}
    &\qquad  - \pi_{\xi/2}(Z)^{-1}
     \pi_{\xi/2}^R(\phi(m) Z).
   \end{align}
   For any $v \in V[<\tfrac{\xi}{2}]$ note that
   $
    Z \Psi_{>-\tfrac{\xi}{2}}
    \left(\Pi_{<\tfrac{\xi}{2}} Z \Psi_{>-\tfrac{\xi}{2}}\right)^{-1}  \cdot v = v + \sum p_i(z) \otimes v_i,
    $
    where each $v_i \in V[\xi/2]$.
    Therefore, since $m \in \g_{\geq 1/2}$, 
    $ \Pi_{<\tfrac{\xi}{2}} \phi(m)
    Z \Psi_{>-\tfrac{\xi}{2}}
    \left(\Pi_{<\tfrac{\xi}{2}} Z \Psi_{>-\tfrac{\xi}{2}}\right)^{-1}  \cdot v = \Pi_{<\tfrac{\xi}{2}} \phi(m) \cdot v$.
    Thus line \eqref{EQ:line1} is equal to
    \begin{align*}
    \pi_{\xi/2}(Z) ^{-1}
   \Pi_{<\tfrac{\xi}{2}} \phi(m) \Psi_{<\tfrac{\xi}{2}}
   \pi_{\xi/2}^R(Z)
    =
    \pi_{\xi/2}(Z) ^{-1}
   \pi_{\xi/2}^R(Z)(\phi(m) Z),
   \end{align*}
   which cancels with \eqref{EQ:line4}.

   Next observe that
     $\left(\Pi_{<\tfrac{\xi}{2}} Z \Psi_{>-\tfrac{\xi}{2}}\right)^{-1}
   \Pi_{<\tfrac{\xi}{2}}  Z \cdot v = v$ for all $v \in V[>-\tfrac{\xi}{2}]$, thus \eqref{EQ:line2} is equal to
   \begin{align*}
    &-  
    \Pi_{>-\tfrac{\xi}{2}} \phi(m) \Psi_{>-\tfrac{\xi}{2}}
    \pi_{\xi/2}(Z)^{-1}
   \pi_{\xi/2}^R(Z).
   \end{align*}
  
  Finally our last observation also gives that \eqref{EQ:line3}
  is equal to \\
  $
    \Pi_{>-\tfrac{\xi}{2}} \phi(m) \Psi_{-\tfrac{\xi}{2}}.
     $
     The lemma now follows.
\end{proof}

\section{Explicit formulas for $L(z)$} \label{S:Lz}
\subsection{Pyramids and coordinates}
Let $N = \dim V$.  Let $\lambda = (\lambda_1 \geq \lambda_2 \geq \dots \geq \lambda_n)$ be the Jordan type of $F = \phi(f) \in \gl(V)$.
We will use a variation of the pyramids from \cite{EK} to define coordinates.
A pyramid is a collection of $1 \times 1$ boxes arranges in rows and columns in the plane.
We make a pyramid of boxes from this partition so that the row lengths are the parts of the partition
and the rows are
arranged symmetrically around
the vertical center line.  Fill the boxes of this pyramid with $1, 2, \dots, N$.  Label the the rows of the pyramid $1, \dots n$ such that the length of row $i$ is $\lambda_i$, and
label the columns with $1, \dots, \lambda_1$, starting at the leftmost box.  Some columns will have
half-integer labels if $\lambda$ has odd and even parts.
For example, if $\lambda = (6,3,3,2)$
then
the filled pyramid could be

\begin{center}
\begin{picture}(120,80)
\put(0,0){\line(1,0){120}}
\put(0,20){\line(1,0){120}}
\put(30,40){\line(1,0){60}}
\put(30,60){\line(1,0){60}}
\put(40,80){\line(1,0){40}}
\put(0,0){\line(0,1){20}}
\put(20,0){\line(0,1){20}}
\put(40,0){\line(0,1){20}}
\put(60,0){\line(0,1){20}}
\put(80,0){\line(0,1){20}}
\put(100,0){\line(0,1){20}}
\put(120,0){\line(0,1){20}}
\put(30,20){\line(0,1){40}}
\put(50,20){\line(0,1){40}}
\put(70,20){\line(0,1){40}}
\put(90,20){\line(0,1){40}}
\put(40,60){\line(0,1){20}}
\put(60,60){\line(0,1){20}}
\put(80,60){\line(0,1){20}}
\put(10,10){\makebox(0,0){{1}}}
\put(30,10){\makebox(0,0){{2}}}
\put(50,10){\makebox(0,0){{3}}}
\put(70,10){\makebox(0,0){{4}}}
\put(90,10){\makebox(0,0){{5}}}
\put(110,10){\makebox(0,0){{6}}}
\put(40,30){\makebox(0,0){{7}}}
\put(60,30){\makebox(0,0){{8}}}
\put(80,30){\makebox(0,0){{9}}}
\put(40,50){\makebox(0,0){{10}}}
\put(60,50){\makebox(0,0){{11}}}
\put(80,50){\makebox(0,0){{12}}}
\put(50,70){\makebox(0,0){{13}}}
\put(70,70){\makebox(0,0){{14}}}
\end{picture}
\end{center}

If $a \in \{1, \dots, N \}$, we let $\row(a)$ be the label of the row in which $a$ occurs in the pyramid, and we let
$\col(a)$ denote the column in which $a$ occurs.
For example for the above pyramid $\row(13)=4$ and $\col(13) = 3$,
whereas $\row(12) = 3$ and $\col(12) = 4.5$.

The
representation theory
of
$\mathfrak{sl}_2$
tells us that
there exists a basis $\{e_1, \dots, e_{N}\}$ of $V$ such that
$F e_a = e_b$ where $\row(a) = \row(b)$ and $\col(b) = \col(a)+1$ if $\col(a) < \lambda_{\row(a)}$, and $F e_a = 0$ if $\col(a) = \lambda_{\row(a)}$.
So with respect to this basis,
\begin{equation} \label{EQ:F}
   F = \sum_{\substack{a,b = 1, \dots, N \\ \row(a) = \row(b) \\
   \col(a) \neq \lambda_1 \\
   \col(b) = \col(a)+1}}  e_{b,a}
\end{equation}

We can now use the pyramid the calculate the weight function
$\delta$ from $\S1$:
\begin{align} \label{EQ:delta}
   \delta(e_a) &= (\lambda_1 + 1)/2 -\col(a), \\
   \delta(e_{a,b}) &= \col(b)-\col(a). \notag
\end{align}

Note that if
\[
  \alpha = \Pi_{<\tfrac{\xi}{2}} \mathbbm{1}_V + F \Psi_{>-\tfrac{\xi}{2}} \in \Hom(V[>-\tfrac{\xi}{2}], V[<\tfrac{\xi}{2}]),
\]
then with respect to the above basis,
\[
   \alpha = \sum_
        {\substack{a,b = 1, \dots, N \\
          \row(a) = \row(b) \\
          \col(b) = \col(a)-1 \\
          \col(a) \neq 1
           }}
       e_{a,b}
       + \sum_{\substack{a=1, \dots, N \\
       \col(a) \neq 1 \\
       \col(a) \neq  \lambda_1
       }}
         e_{a,a},
\]
where we have swapped $a,b$ from where they were in \eqref{EQ:F}.
A straightforward calculation shows that
\begin{equation} \label{EQ:alphainv}
 \alpha^{-1} = \sum_{
 \substack{
    a,b = 1, \dots, N \\
    \lambda_{\row(a)} = \lambda_1 \\
    \row(a) = \row(b) \\
    \col(a) < \col(b)
   }}
   -(-1)^{\col(a) + \col(b)} e_{a,b}
   +
 \sum_{
 \substack{
    a,b = 1, \dots, N \\
    \lambda_{\row(a)} < \lambda_1 \\
    \row(a) = \row(b) \\
    \col(a) \geq \col(b)
   }}
   (-1)^{\col(a) + \col(b)} e_{a,b}.
\end{equation}

We now use these calculations to find a formula
for
\begin{align} \label{EQ:inv}
&\left(\Pi_{<\tfrac{\xi}{2}}
Y
\Psi_{>-\tfrac{\xi}{2}} \right)^{-1}
 \\
\notag
& \qquad =
\left(
\Pi_{<\tfrac{\xi}{2}} \mathbbm{1} + F +
\sum_{u_i \in \g_{\leq 1/2}} z^{\delta(u_i)-1} u_i \otimes U^i + z^{-1} D
\Psi_{>-\tfrac{\xi}{2}} \right)^{-1}.
\end{align}

Let
\begin{equation} \label{EQ:Yz1}
\bar Y = \Pi_{< \tfrac{\xi}{2}}  \sum_{u_i \in \g_{\leq 1/2}} z^{\delta(u_i)-1} u_i \otimes U^i + z^{-1} D
\Psi_{> -\tfrac{\xi}{2}}.
\end{equation}
In general we shall use $\bar Y_{a,b}$ to refer to the
coefficient of $e_{a,b}$ in
\[
 \sum_{u_i \in \g_{\leq 1/2}} z^{\delta(u_i)-1} u_i \otimes U^i + z^{-1} D,
\]
though the matrix $\bar Y$ only involves a subset of the $\bar Y_{a,b}$'s.
It will also be useful to note that
\begin{equation} \label{EQ:Yab}
  \bar Y_{a,b} = z^{\col(a) - \col(b) - 1} x_{a,b}
\end{equation}
for some $x_{a,b} \in U(\g)$,
which means that $\bar Y$ is a polynomial in $z^{-1/2}$ with no constant term.

Now we have that \eqref{EQ:inv}
equals
\[
(\alpha + \bar Y)^{-1}
=(1+ \alpha^{-1} \bar Y)^{-1} \alpha^{-1}
= \left(\sum_{m=0}^\infty (-\alpha^{-1} \bar Y)^m \right) \alpha^{-1}.
\]

We shall first find a formula for $-\alpha^{-1} \bar Y$.
From \eqref{EQ:Yz1} since $u^i \in \g_{\geq -1/2}$ we have that
\[
\bar Y =
\sum_{\substack{a,b = 1, \dots, N \\ \col(a) \leq \col(b)+1/2 \\
  \col(a) > 1 \\
  \col(b) < \lambda_1}} \bar Y_{a,b} e_{a,b}.
\]
Thus
\begin{align*}
 &-\alpha^{-1} \bar Y = \\
 &\qquad
 \bigg(
  \sum_{
 \substack{
    c,d = 1, \dots, N \\
    \lambda_{\row(c)} = \lambda_1 \\
    \row(c) = \row(d) \\
    \col(c) < \col(d)
   }}
   -(-1)^{\col(c) + \col(d)} e_{c,d}
   +
 \sum_{
 \substack{
    c,d = 1, \dots, N \\
    \lambda_{\row(c)} < \lambda_1 \\
    \row(c) = \row(d) \\
    \col(c) \geq \col(d)
   }}
   (-1)^{\col(c) + \col(d)} e_{c,d} \bigg) \\
   &\qquad \qquad \times
\sum_{\substack{a,b = 1, \dots, N \\ \col(a) \leq \col(b)+1/2 \\
  \col(a) > 1 \\
  \col(b) < \lambda_1}} \bar Y_{a,b} e_{a,b} \\
&\qquad \qquad =
 \sum_{\substack{
 c,b,d = 1, \dots, N \\
 \lambda_{\row(c)} = \lambda_1 \\
 \row(c) = \row(d) \\
 \col(c) < \col(d) \leq \col(b) +1/2 \\
  \col(b) < \lambda_1
 }}
 (-1)^{\col(c) + \col(d)} \bar Y_{d,b} e_{c,b} \\
 &\qquad \qquad \qquad \qquad
 +
 \sum_{\substack{
 c,b,d = 1, \dots, N \\
 \lambda_{\row(c)} < \lambda_1 \\
 \row(c) = \row(d) \\
 \col(c) \geq \col(d) > 1\\
 \col(d) \leq \col(b) +1/2 \\
  \col(b)< \lambda_1
 }}
 -(-1)^{\col(c) + \col(d)} \bar Y_{d,b} e_{c,b}
\end{align*}
Thus
\[
 (-\alpha^{-1} \bar Y)_{c,b} =
 \sum_{\substack{ d = 1, \dots, N\\ \
     \row(c) = \row(d) \\
    \col(c) < \col(d) \\
    \col(d) \leq \col(b) + 1/2 }}
 (-1)^{\col(c)+\col(d)} \bar Y_{d,b}
\]
if $\lambda_{\row(c)} = \lambda_1$ where $\col(b) < \lambda_1$,
and
\[
 (-\alpha^{-1} \bar Y)_{c,b} =
 \sum_{\substack{ d = 1, \dots, N\\ \
     \row(c) = \row(d) \\
    \col(c) \geq \col(d) > 1 \\
    \col(d) \leq \col(b) + 1/2 }}
  -(-1)^{\col(c)+\col(d)} \bar Y_{d,b}
\]
if $\lambda_{\row(c)} < \lambda_1$ where $\col(b) < \lambda_1$.

For $a_1, \dots, a_k \in \{1, \dots, N\}$ let
$n(a_1, \dots, a_k) =  | \{ a_i \mid \lambda_{\row(a_i)} < \lambda_1 \}|$,
and let $\col(a_1, \dots, a_k) = \sum_{i=1}^k \col(a_i)$.

So we have that
\begin{align*}
 &((-\alpha^{-1} \bar Y)^m)_{c,b}  =
 \sum (-\alpha^{-1} \bar Y)_{c,a_1} (-\alpha^{-1} \bar Y)_{a_1, a_2} \dots (-\alpha^{-1} \bar Y)_{a_{m-1},b} \\
  &\quad =
 \sum (-1)^{\col(c, a_1 ,\dots , a_{m-1} , d_1 , \dots ,d_{m})+n(c, a_1, \dots, a_{m-1}) } \bar Y_{d_1, a_1} \bar Y_{d_2, a_2} \dots \bar Y_{d_m, a_m}
\end{align*}
where in the last sum we are summing over all
$a_1, \dots, a_m, d_1, \dots, d_m \in \{1, \dots, N\}$ such that
\begin{itemize}
 \item $\row(d_1) = \row(c)$,
 \item $a_m = b$,
 \item $\col(d_i) > 1$ for all $i$,
 \item $\col(a_i) < \lambda_1$ for all $i$,
  \item $\col(d_i) \leq \col(a_i)+1/2$ for all $i$,
  \item $\row(d_{i+1}) = \row(a_i)$ for all $i < m$,
  \item $\col(a_i) < \col(d_{i+1})$ if $\lambda_{\row(a_i)} = \lambda_1$ for all $i < m$,
  \item $\col(a_i) \geq \col(d_{i+1})$ if $\lambda_{\row(a_i)} < \lambda_1$ for all $i < m$,
  \item $\col(c) < \col(d_{1})$ if $\lambda_{\row(c)} = \lambda_1$,
  \item $\col(c) \geq \col(d_{1})$ if $\lambda_{\row(c)} < \lambda_1$.
\end{itemize}

Next we note that for any matrix $A$ with rows and columns indexed by the same set as the rows of $\alpha^{-1}$, that
\begin{align*}
 &(A \alpha^{-1})_{c,d} = \sum_{\substack{ b = 1, \dots, N \\ \col(b) < \lambda_1}}
 A_{c,b} \alpha^{-1}_{b,d}  \\
 &\quad=
 \sum_{\substack{
  b = 1, \dots, N \\
  \col(b) < \lambda_1 \\
 \lambda_{\row(b)} = \lambda_1 \\
    \row(b) = \row(d) \\
    \col(d) > \col(b) }}
    -(-1)^{\col(b) + \col(d)} A_{c,b}
     +
 \sum_{\substack{
  b = 1, \dots, N \\
  \col(b) < \lambda_1 \\
 \lambda_{\row(b)} < \lambda_1 \\
    \row(b) = \row(d) \\
    \col(d) \leq \col(b) }}
    (-1)^{\col(b) + \col(d)} A_{c,b}
\end{align*}
Thus
\begin{align} \label{EQ:Ypower}
  &(-\alpha^{-1} \bar Y)^m \alpha^{-1})_{c,d} \\ \notag
  & \quad =
 \sum -(-1)^{\col(c, a_1,\dots , a_m , d_1 , \dots ,d_{m})+n(c, a_1, \dots, a_m) +\col(d)} \bar Y_{d_1, a_1} \bar Y_{d_2, a_2} \dots \bar Y_{d_m, a_m}
 \end{align}
and we are summing over all
$a_1, \dots, a_m, d_1, \dots, d_m \in \{1, \dots, N\}$ such that
\begin{itemize}
 \item $\row(d_1) = \row(c)$,
  \item $\row(a_m) = \row(d)$,
 \item $\col(d_i) > 1$ for all $i$,
 \item $\col(a_i) < \lambda_1$ for all $i$,
  \item $\col(d_i) \leq \col(a_i)+1/2$ for all $i$,
  \item $\row(d_{i+1}) = \row(a_i)$ for all $i < m$,
  \item $\col(a_i) < \col(d_{i+1})$ if $\lambda_{\row(a_i)} = \lambda_1$ for all $i < m$,
  \item $\col(a_i) \geq \col(d_{i+1})$ if $\lambda_{\row(a_i)} < \lambda_1$ for all $i < m$,
  \item $\col(c) < \col(d_{1})$ if $\lambda_{\row(c)} = \lambda_1$,
  \item $\col(c) \geq \col(d_{1})$ if $\lambda_{\row(c)} < \lambda_1$,
  \item $\col(a_m) < \col(d)$ if $\lambda_{\row(d)} = \lambda_1$,
  \item $\col(a_m) \geq \col(d)$ if $\lambda_{\row(d)} < \lambda_1$.
\end{itemize}

Now we can use these results to find a formula for the entries of
\[
L_R(z) = \left(
\Pi_{<\tfrac{\xi}{2}}
Y
\Psi_{>-\tfrac{\xi}{2}} \right)^{-1}
\Pi_{<\tfrac{\xi}{2}}
Y
\Psi_{-\tfrac{\xi}{2}}.
\]
Note that
\[
\Pi_{<\tfrac{\xi}{2}}
Y
\Psi_{-\tfrac{\xi}{2}}
 =
 \sum_{\substack{
 b= 1, \dots, N \\
 \col(b) = \lambda_1}} e_{b,b}+
 \sum_{\substack{
  a,b = 1, \dots, N \\
  \col(b) = \lambda_1 \\
  \col(a) > 1}}
 \bar Y_{a,b} e_{a,b}.
\]
Let $c,d \in \{1, \dots, N\}$ such that
$\col(c) \neq \lambda_1$ and $\col(d) = \lambda_1$.
So, from \eqref{EQ:Ypower}, we get that
\[
\left((-\alpha^{-1} \bar Y)^m \alpha^{-1}
\Pi_{<\tfrac{\xi}{2}}
Y
\Psi_{-\tfrac{\xi}{2}}\right)_{c,d} = A + B,
\]
where
\[
A= - \sum (-1)^{\col(c, a_1,\dots, a_m, d_1 , \dots ,d_{m})+n(c, a_1, \dots, a_m) +\lambda_1} \bar Y_{d_1, a_1} \bar Y_{d_2, a_2} \dots \bar Y_{d_m, a_m}
\]
where we are summing over all
$a_1, \dots, a_m, d_1, \dots, d_m \in \{1, \dots, N\}$ such that
\begin{itemize}
 \item $\row(d_1) = \row(c)$,
  \item $\row(a_m) = \row(d)$,
 \item $\col(d_i) > 1$ for all $i$,
 \item $\col(a_i) < \lambda_1$ for all $i$,
  \item $\col(d_i) \leq \col(a_i)+1/2$ for all $i$,
  \item $\row(d_{i+1}) = \row(a_i)$ for all $i < m$,
  \item $\col(a_i) < \col(d_{i+1})$ if $\lambda_{\row(a_i)} = \lambda_1$ for all $i < m$,
  \item $\col(a_i) \geq \col(d_{i+1})$ if $\lambda_{\row(a_i)} < \lambda_1$ for all $i < m$,
  \item $\col(c) < \col(d_{1})$ if $\lambda_{\row(c)} = \lambda_1$,
  \item $\col(c) \geq \col(d_{1})$ if $\lambda_{\row(c)} < \lambda_1$,
\end{itemize}
and
\[
B= -\sum (-1)^{\col(c, a_1, \dots, a_{m}, d_1, \dots,d_{m+1})+n(c, a_1, \dots, a_{m}) } \bar Y_{d_1, a_1} \bar Y_{d_2, a_2} \dots \bar Y_{d_{m+1}, a_{m+1}}
\]
where we are summing over all
$a_1, \dots, a_{m+1}, d_1, \dots, d_{m+1} \in \{1, \dots, N\}$ such that
\begin{itemize}
 \item $\row(d_1) = \row(c)$,
  \item $\row(a_{m+1}) = \row(d)$,
 \item $\col(d_i) > 1$ for all $i$,
 \item $\col(a_i) < \lambda_1$ for all $i \leq m$,
  \item $\col(a_{m+1}) = \lambda_1$,
  \item $\col(d_i) \leq \col(a_i) +1/2$ for all $i$,
  \item $\row(d_{i+1}) = \row(a_i)$ for all $i \leq m$,
  \item $\col(a_i) < \col(d_{i+1})$ if $\lambda_{\row(a_i)} = \lambda_1$ for all $i \leq m$,
  \item $\col(a_i) \geq \col(d_{i+1})$ if $\lambda_{\row(a_i)} \leq \lambda_1$ for all $i \leq m$,
  \item $\col(c) < \col(d_{1})$ if $\lambda_{\row(c)} = \lambda_1$,
  \item $\col(c) \geq \col(d_{1})$ if $\lambda_{\row(c)} < \lambda_1$,
\end{itemize}

Recall that $\bar Y_{a,b}$ is a monomial in $z^{-1/2}$ of degree $1+\col(b)-\col(a)$.
So we have the following theorem,
if we note that the monomials in the above $B$ term fits the form of the monomials in the theorem since if $s=a_{m+1}$, then
$1 = (-1)^{n(a_{m+1})+\lambda_1}$ since $a_{m+1} = \lambda_1$.
\begin{Theorem} \label{T:LzRexplicit}
    Let $p$ be a positive integer or half-integer.  Then
   The $z^{-p}$ term of $L_R(z)_{c,d}$ is
\[
-\sum (-1)^{\col(c, a_1, \dots, a_{s}, d_1, \dots,d_{s})+n(c, a_1, \dots, a_{s}) +\lambda_1}
\bar Y_{d_1, a_1} \bar Y_{d_2, a_2} \dots \bar Y_{d_{s}, a_{s}}
\]
where we are summing over all
$a_1, \dots, a_{s}, d_1, \dots, d_{s} \in \{1, \dots, N\}$ such that
\begin{itemize}
 \item $\row(d_1) = \row(c)$,
  \item $\row(a_{s}) = \row(d)$,
 \item $\col(d_i) > 1$ for all $i$,
 \item $\col(a_i) < \lambda_1$ for all $i < s$,
  \item $\col(d_i) \leq \col(a_i)+1/2$ for all $i$,
  \item $\row(d_{i+1}) = \row(a_i)$ for all $i < s$,
  \item $\col(a_i) < \col(d_{i+1})$ if $\lambda_{\row(a_i)} = \lambda_1$ for all $i < s$,
  \item $\col(a_i) \geq \col(d_{i+1})$ if $\lambda_{\row(a_i)} < \lambda_1$ for all $i < s$,
  \item $\col(c) < \col(d_{1})$ if $\lambda_{\row(c)} = \lambda_1$,
  \item $\col(c) \geq \col(d_{1})$ if $\lambda_{\row(c)} < \lambda_1$,
  \item $\col(a_1)-\col(d_1) + \dots + \col(a_s) - \col(d_s) -s = p$.
\end{itemize}
\end{Theorem}

Finally
we have that
\[
\Pi_{\tfrac{\xi}{2}} Y \Psi_{>\tfrac{-\xi}{2}}
 =
 \sum_{\substack{
 b= 1, \dots, N \\
 \col(b) = 1}} e_{b,b}+
 \sum_{\substack{
  a,b = 1, \dots, N \\
  \col(a) = 1 \\
  \col(b) < \lambda_1}}
 \bar Y_{a,b} e_{a,b}.
\]
So if $c,d \in \{1, \dots, N\}$ where $\col(c) = 1$ and $\col(d) = \lambda_1$, then we calculate that
\[
\left( \Pi_{\tfrac{\xi}{2}} Y \Psi_{>\tfrac{-\xi}{2}}
(-\alpha^{-1} \bar Y)^m \alpha^{-1}
\Pi_{<\tfrac{\xi}{2}}
Y
\Psi_{-\tfrac{\xi}{2}}\right)_{c,d} = L + M +P+Q,
\]
where
\[
L= \sum (-1)^{\col(a_1,\dots, a_m, d_1 , \dots ,d_{m})+n(a_1, \dots, a_m) +\lambda_1} \bar Y_{d_1, a_1} \bar Y_{d_2, a_2} \dots \bar Y_{d_m, a_m}
\]
where we are summing over all
$a_1, \dots, a_m, d_1, \dots, d_m \in \{1, \dots, N\}$ such that
\begin{itemize}
 \item $\row(d_1) = \row(c)$,
  \item $\row(a_m) = \row(d)$,
 \item $\col(d_i) > 1$ for all $i$,
 \item $\col(a_i) < \lambda_1$ for all $i$,
  \item $\col(d_i) \leq \col(a_i)+1/2$ for all $i$,
  \item $\row(d_{i+1}) = \row(a_i)$ for all $i < m$,
  \item $\col(a_i) < \col(d_{i+1})$ if $\lambda_{\row(a_i)} = \lambda_1$ for all $i < m$,
  \item $\col(a_i) \geq \col(d_{i+1})$ if $\lambda_{\row(a_i)} < \lambda_1$ for all $i < m$,
\end{itemize}
and
\[
M= \sum (-1)^{\col(a_1, \dots, a_{m}, d_1, \dots,d_{m+1})+n(a_1, \dots, a_{m}) } \bar Y_{d_1, a_1} \bar Y_{d_2, a_2} \dots \bar Y_{d_{m+1}, a_{m+1}}
\]
where we are summing over all
$a_1, \dots, a_{m+1}, d_1, \dots, d_{m+1} \in \{1, \dots, N\}$ such that
\begin{itemize}
 \item $\row(d_1) = \row(c)$,
  \item $a_{m+1} = d$,
 \item $\col(d_i) > 1$ for all $i$,
 \item $\col(a_i) < \lambda_1$ for all $i \leq m$,
  \item $\col(a_{m+1}) = \lambda_1$,
  \item $\col(d_i) \leq \col(a_i) +1/2$ for all $i$,
  \item $\row(d_{i+1}) = \row(a_i)$ for all $i \leq m$,
  \item $\col(a_i) < \col(d_{i+1})$ if $\lambda_{\row(a_i)} = \lambda_1$ for all $i \leq m$,
  \item $\col(a_i) \geq \col(d_{i+1})$ if $\lambda_{\row(a_i)} < \lambda_1$ for all $i \leq m$,
\end{itemize}
and
\[
P=  \sum (-1)^{\col(a_0, a_1,\dots, a_m, d_0, d_1, \dots ,d_{m})+n(a_0, a_1, \dots, a_m) +\lambda_1} \bar Y_{d_0, a_0} \bar Y_{d_1, a_1} \dots \bar Y_{d_m, a_m}
\]
where we are summing over all
$a_0, \dots, a_m, d_0, \dots, d_m \in \{1, \dots, N\}$ such that
\begin{itemize}
 \item $d_0 = c$,
  \item $\row(a_m) = \row(d)$,
 \item $\col(d_i) > 1$ for all $i \geq 1$,
   \item $\col(d_0) = 1$,
 \item $\col(a_i) < \lambda_1$ for all $i$,
  \item $\col(d_i) \leq \col(a_i)+1/2$ for all $i$,
  \item $\row(d_{i+1}) = \row(a_i)$ for all $i < m$,
  \item $\col(a_i) < \col(d_{i+1})$ if $\lambda_{\row(a_i)} = \lambda_1$ for all $i < m$,
  \item $\col(a_i) \geq \col(d_{i+1})$ if $\lambda_{\row(a_i)} < \lambda_1$ for all $i < m$,
\end{itemize}
and
\begin{align*}
& Q=  \\ \quad & \sum (-1)^{\col(a_0, a_1, \dots, a_{m}, d_0, d_1, \dots,d_{m+1})+n(a_0, a_1, \dots, a_{m}) } \bar Y_{d_0, a_0} \bar Y_{d_1, a_1} \dots \bar Y_{d_{m+1}, a_{m+1}}
\end{align*}
where we are summing over all
$a_1, \dots, a_{m+1}, d_1, \dots, d_{m+1} \in \{1, \dots, N\}$ such that
\begin{itemize}
 \item $d_0 = c$,
  \item $a_{m+1} = d$,
 \item $\col(d_i) > 1$ for all $i \geq 1$,
 \item $\col(a_i) < \lambda_1$ for all $i \leq m$,
  \item $\col(a_{m+1}) = \lambda_1$,
  \item $\col(d_i) \leq \col(a_i) +1/2$ for all $i$,
  \item $\row(d_{i+1}) = \row(a_i)$ for all $i \leq m$,
  \item $\col(a_i) < \col(d_{i+1})$ if $\lambda_{\row(a_i)} = \lambda_1$ for all $i \leq m$,
  \item $\col(a_i) \geq \col(d_{i+1})$ if $\lambda_{\row(a_i)} < \lambda_1$ for all $i \leq m$,
\end{itemize}

This all leads to a (somewhat) explicit formula for the coefficients of
$L(z) =
\Pi_{\tfrac{\xi}{2}} Y \Psi_{-\tfrac{\xi}{2}} -
\Pi_{\tfrac{\xi}{2}} Y \Psi_{> -\tfrac{\xi}{2}}
\left(\Pi_{< \tfrac{\xi}{2}} Y \Psi_{> -\tfrac{\xi}{2}}\right)^{-1}
\Pi_{< \tfrac{\xi}{2}} Y \Psi_{-\tfrac{\xi}{2}}$:
\begin{Theorem} \label{T:Lzexplicit}
    Let $p$ be a positive integer.  Then
   The $z^{-p}$ term of $L(z)_{c,d}$ is
\[
-\sum (-1)^{\col(a_1, \dots, a_{s}, d_1, \dots,d_{s})+n(a_1, \dots, a_{s}) + \lambda_1 }
\bar Y_{d_1, a_1} \bar Y_{d_2, a_2} \dots \bar Y_{d_{s}, a_{s}}
\]
where we are summing over all
$a_1, \dots, a_{s}, d_1, \dots, d_{s} \in \{1, \dots, N\}$ such that
\begin{itemize}
 \item $\row(d_1) = \row(c)$,
  \item $\row(a_{s}) = \row(d)$,
 \item $\col(d_i) > 1$ for all $i > 1$,
 \item $\col(a_i) < \lambda_1$ for all $i < s$,
  \item $\col(d_i) \leq \col(a_i)+1/2$ for all $i$,
  \item $\row(d_{i+1}) = \row(a_i)$ for all $i < s$,
  \item $\col(a_i) < \col(d_{i+1})$ if $\lambda_{\row(a_i)} = \lambda_1$ for all $i < s$,
  \item $\col(a_i) \geq \col(d_{i+1})$ if $\lambda_{\row(a_i)} < \lambda_1$ for all $i < s$,
  \item $\col(a_1)-\col(d_1) + \dots + \col(a_s) - \col(d_s) +s = p$.
\end{itemize}
\end{Theorem}
For completeness we note by the formula for
$\alpha^{-1}$ from \eqref{EQ:alphainv} that
the constant term of $L_{c,d}(z)$ is $-\delta_{\row(c), \row(d)} (-1)^{\lambda_1}$.

The following lemmas which will prove useful for proving Theorems \ref{T:Lk} and \ref{T:LRk}.

\begin{Lemma} \label{L:aanotappear}
 Let $a \in \{ 1, \dots, N\} $ such that $\col(a) = 1$ or $\col(a) = \lambda_1$.

 If $\col(a) = 1$, then
 let  $i,j \in \{1, \dots, n\}$ such that $i \neq \row(a)$.
 Then $\bar Y_{a,b}$ does not appear in any coefficient of $L(z)_{i,j}$ in the expression for $L(z)$ from Theorem \ref{T:Lzexplicit} for any $b$.

 If $\col(a) = \lambda_1$, then
 let  $i,j \in \{1, \dots, n\}$ such that $j \neq \row(a)$.
 Then $\bar Y_{b,a}$ does not appear in any coefficient of $L(z)_{i,j}$ in the expression for $L(z)$ from Theorem \ref{T:Lzexplicit} for any $b$.
\end{Lemma}
\begin{proof}
Let
$\bar Y_{d_1, a_1} \bar Y_{d_2, a_2} \dots \bar Y_{d_{s}, a_{s}}$ be one of monomials from Theorem \ref{T:Lzexplicit}.
Note that $\lambda_{\row(a)} = \lambda_1$ in these cases.
Suppose that $\col(a) = 1$.
Since $\row(a) \neq i$, $d_1 \neq a$, so $\bar Y_{a,b}$
does not occur in the first position.
Now suppose that $d_k = a$ for some $k \geq 2$.
So $\col(a_{k-1}) < \col(a)$, which is not possible since $\col(a) = 1$.

Now suppose that $\col(a) = \lambda_1$.  Since $\row(a) \neq j$, $a_s \neq a$.
So $\bar Y_{b,a}$ does not occur in the last position.  Now suppose that
$a_k = a$ for some $k < \lambda_1$.  So $\col(d_{k+1}) > \col(a)$,
which cannot happen since $\col(a) = \lambda_1$.
\end{proof}

The same proof also proves the following:
\begin{Lemma} \label{L:aanotappear2}
 Let $a \in \{ 1, \dots, N\} $ such that $\col(a) = \lambda_1$.
 Let  $i, j \in \{1, \dots, n\}$ such that $j \neq \row(a)$.
 Then $\bar Y_{a,a}$ does not appear in any coefficient of $L_R(z)_{i,j}$ in the expression for $L_R(z)$ from Theorem \ref{T:LzRexplicit}.
\end{Lemma}

\subsection{Proof of Theorem \ref{T:main}}
\begin{proof}
Let $w_i$ be a highest weight vector of $V$ such that $\lambda_i < \lambda_1$ and
$\lambda_i \geq \lambda_j$ for all $j$ such that $\lambda_j < \lambda_1$.
Note that the weight of $w_i$ is $(\lambda_i-1)/2$,
so the condition $b > \xi/2-l$ from the theorem is equivalent to $b > (\lambda_1-\lambda_i)/2$.
Let $c, d \in \{1, \dots, N\}$ such that $\row(c) = i$, $\lambda_{\row(d)} = \lambda_1$, and $\col(c) = \lambda_i$.
%Since $\m$ is generated by $\g_1$ and $\g_{1/2}$,
We need to prove that
$[m, L_R(z)]_{c,d}$ is a polynomial of degree less than $(\lambda_1 - \lambda_i)/2$ for all $m \in \g_{\geq 1/2}$.
Let $m \in \g_{\geq 1/2}$.
So $\phi(m)$ is a linear combination of $E_{a,b}$'s such that $\col(b)-\col(a) = 1$ or $\col(b) - \col(a) = 1/2$.
Let $E_{a,b}$ be such a term.  By Lemma \ref{L:help}, we need to prove that
$(E_{a,b} L_R(z) \bar 1)_{c,d}$ is a polynomial in $z^{-1/2}$ of degree at most
$(\lambda_1 - \lambda_i) /2$.  For this not to be zero, we must have that $a=c$, and we need to prove that
$(L_R(z) \bar 1)_{b,d}$
is a polynomial
of degree at most $(\lambda_1 - \lambda_i)/2$.
Since $\col(b) > \col(a)$ and $\lambda_i$ is maximal, we must have that $\lambda_{\row(b)} = \lambda_1$.
Let
$\bar Y_{d_1, a_1} \bar Y_{d_2, a_2} \dots \bar Y_{d_{s}, a_{s}}$ be a monomial in the expression of
$(L_R(z) \bar 1)_{b,d}$
from Theorem \ref{T:LzRexplicit}.
So $\col(a_1) > \col(b) > \col(c)$ which implies
that for each index $i$ in the monomial that
$\lambda_{\row(a_i)} = \lambda_{\row(d_i)} = \lambda_1$.
So $\col(a_i) \leq \col(d_i)$ and $\col(d_i) < \col(a_{i+1})$ for all $i$ in the monomial.
So $d_1-a_1 - 1 + d_2-a_2-1 + \dots d_s - a_s - 1 \leq (\lambda_1-\lambda_i)/2$,  thus the monomial has degree at most $(\lambda_1 - \lambda_i)/2$.
\end{proof}

\subsection{Pyramids for $L_k(z)$}
Let $k$ be a highest weight for the $\mathfrak{sl}_2$-triple $(f,2x,e)$ and the module $V$.
We consider the pyramid for
$L_k(z)$ to be a subpyramid of the pyramid of $L(z$), with the same column and row labels.
So if $k < \xi/2$, then the pyramid for $L_k(z)$ will be the pyramid for $L_k(z)$ with some columns removed, and its first column's label won't be $1$.
Let $s_k$ be the column label of the first column of the pyramid for $L_k(z)$, let $e_k$ be the column label of the last column, and let $r_k$ be the length of the longest row in
the pyramid for $L_k(z)$.
Note that
\begin{align} \label{EQ:sk}
 s_k &= \frac{\lambda_1+1}{2}-k, \\ \notag
  e_k &= \frac{\lambda_1+1}{2}+k, \text{ and } \\ \notag
 r_k &= 2k+1.
\end{align}

\section{Realizations of $\sp_N$ and $\so_N$}

\subsection{Choosing coordinates}
We need realizations for $\sp_N$ and $\so_N$ which are convenient to work with.  While the commonly used $f_{i,j} = e_{i,j} \pm e_{-j,-i}$
basis works well in most cases (see eg \cite{B1}), we can account for extra cases if we switch to the approach from \cite[$\S$2.1]{PrT}.

For this approach, we work with $\gl(V)$, where $\dim V = N$.
Choose $\eps \in \{\pm 1\}$.
Let $J$ be a symmetric form on $V$ if $\eps = 1$ and let $J$
be a skew-symmetric form on $V$ if $\eps = -1$.
Let $\sigma: \gl(V) \to \gl(V)$ be defined via
$\sigma(X) = -J^{-1} X^T J$, where we are considering
$X$ and $J$ to be matrices with respect to any basis of $V$.  Now $\sigma$ is an involution and $\sigma$ does not depend on the choice of basis.  Let $\g = \{X \in \gl(V) \mid \sigma(X) = X\}$.
Now $\g \cong \so_N$ if $\eps = 1$ and $\g \cong \sp_N$ if $\eps = -1$.

Let $f \in \g$ be a nilpotent element in $\g$.
Let $\lambda_1 \geq \lambda_2 \dots \geq \lambda_n$ be the Jordan type of $f$.  By $\mathfrak{sl}_2$ representation theory, there exists
a set of highest weight vectors $\{w_1, \dots, w_n\}$ such that
$\{f^s w_i \mid 1 \leq i \leq n, 0 \leq s < \lambda_i\}$ is a basis of $\g$.
Let $V[i] = \{f^s w_i\}$ for $i = 1, \dots, n$.

The following is \cite[Lemma 1]{PrT}:
\begin{Lemma}
 There exists an involution $i \mapsto i'$ on $\{1, \dots, n\}$ such that
 \begin{enumerate}
  \item $\lambda_i = \lambda_i'$ for all $i$,
  \item $(V[i], V[j]) = 0$ if $i \neq j'$,
  \item $i = i'$ if and only if $\eps (-1)^{\lambda_i} = -1$.
 \end{enumerate}
\end{Lemma}
Note that the third condition says $i=i'$ if $\eps = 1$ and $\lambda_i$ is odd, or $\eps = -1$ and $\lambda_i$ is even.
 Furthermore we can index the $w_i$'s so that
 $i' \in \{i-1, i, i+1\}$ for all $i$.
 Now, as explained in \cite[$\S$2.1]{PrT},
 $(f^{\lambda_i-1} w_i, f^s w_i) = 0$ for all $i$ and $s > 0$,
 and the vectors $\{w_i\}$ can be normalized so that
 $(w_i, f^{\lambda_i-1} w_{i'}) = 1$ whenever $i \leq i'$.

 For the following lemma we
consider $J$ to be a matrix with respect to the basis
$\{f^s w_i \mid 1 \leq i \leq n, 0 \leq s < \lambda_i\}$,
and we let $s_i' = \lambda_i -1 -s$ for all $s \in \mathbb{Z}$.
When the $i$ is clear from context, we write $s' = s_i'$.
 \begin{Lemma} \label{L:J}
     The following holds for all $i \in \{1, \dots, n\}$ and \\
     $s \in \{0, \dots, \lambda_i - 1\}$:
    \begin{enumerate}
     \item $(f^s w_i, f^{s'} w_{i'}) = (-1)^{s}$ if $i \leq i'$,
     \item $J f^{s'} w_{i'} = (-1)^s f^s w_i$ if $i \leq i'$,
     \item $(f^s w_i, f^{s'} w_{i'}) = \eps (-1)^{s'}$ if $i > i'$,
     \item $J f^{s'} w_{i'} = \eps (-1)^{s'} f^s w_i$ if $i > i'$.
    \end{enumerate}
 \end{Lemma}
 \begin{proof}
 The first point follows from
 $(w_i, f^{\lambda_i-1} w_{i'}) = 1$, $(v,w) = v^TJw$ for all $v,w \in V$, and $JX+X^T J = 0$ for all $X \in \g$.

 For the third point, assume that $i > i'$.  Then by the first point
 $(f^{s'} w_{i'}, f^s w_i) = (-1)^{s'}$ which implies
 $(f_s w_i, f^{s'}w_{i'}) = \eps (f^{s'} w_{i'}, f^s w_i) = \eps(-1)^{s'}$.

 The second and forth points follow from the fact that the matrix of $J$ with respect to a basis $\{v_1, \dots, v_N\}$ is defined via
 $J_{i,j} = (v_i, v_j)$.
 \end{proof}

We define $e_{i,s}^{j,t} \in \gl(V)$ via
\[
e_{i,s}^{j,t} f^p w_k = \delta_{i,k} \delta_{s,p} f^t w_j.
\]
Note that $e_{i,s}^{j,t} = 0$ if $t \geq \lambda_j$.
Now $\{e_{i,s}^{j,t} \mid 1 \leq i,j \leq n, 0 \leq s < \lambda_i, 0 \leq t < \lambda_j\}$ is a basis of $\gl(V)$.
Thus $\g = \{X \in \gl(V) \mid \sigma(X) = X\}$ is spanned by
$\{e_{i,s}^{j,t} + \sigma(e_{i,s}^{j,t}) \}$.

For $i \in \{1, \dots, n\}$ and $s \in \Z$  define
\[
   \eta_{i \leq i'}(s) = \begin{cases}
                          (-1)^s & \text{ if } i \leq i'; \\
                          \eps(-1)^{s'} & \text{ if } i > i'.
                         \end{cases}
\]
So by Lemma \ref{L:J}
$J f^{s'} w_i' = \eta_{i \leq i'}(s) f^s w_i$ in all cases.

\begin{Lemma}
   Let $i,j \in \{1, \dots n\}, 0 \leq s < \lambda _i, 0 \leq t < \lambda_j$.  Then
  \[
   \sigma(e_{i,s}^{j,t}) = -\eps \eta_{j \leq j'}(t) \eta_{i' \leq i}(s') e_{j',t'}^{i',s'}.
   \]
\end{Lemma}
\begin{proof}
Let $k \in \{1, \dots, n\}$ and let $p \in \{0, \dots, \lambda_k-1\}$.  Now
\begin{align*}
 \sigma(e_{i,s}^{j,t})f^{p'}w_{k'} &=
   -J^{-1} (e_{i,s}^{j,t})^T J f^{p'} w_{k'} =
   -\eps J e_{j,t}^{i,s} J f^{p'} w_{k'} \\
   & = -\eps \eta_{k \leq k'}(p) J e_{j,t}^{i,s} f^{p} w_{k}
   = -\eps \eta_{k \leq k'}(p) \delta_{j,k} \delta_{t, p} J  f^{s} w_{i} \\
   &= -\eps \eta_{k \leq k'}(p) \eta_{i' \leq i}(s') \delta_{j,k} \delta_{t, p} f^{s'} w_{i'} \\
   &= -\eps \eta_{j \leq j'}(t) \eta_{i' \leq i}(s') \delta_{j,k} \delta_{t, p} f^{s'} w_{i'}.
 \end{align*}
 The lemma now follows.
\end{proof}

We define
\[
  f_{i,s}^{j, t} = e_{i,s}^{j,t} + \sigma(e_{i,s}^{j,t})
 = e_{i,s}^{j,t} -\eps \eta_{j \leq j'}(t) \eta_{i' \leq i}(s') e_{j',t'}^{i',s'}.
\]

\begin{Lemma} \label{L:f0}
Let $i,j \in \{1, \dots, n\}$, let $s \in \{0, \dots, \lambda_i-1\}$, and let $t \in \{0, \dots, \lambda_j-1\}$.  Then
 $f_{i,s}^{j,t} = 0$ if and only if $j=i'$, $t=s'$, and $\eps = 1$.
\end{Lemma}
\begin{proof}
Clearly in order for $f_{i,s}^{j,t}$ to be zero we need
$j=i'$, $t=s'$ and
$\eps \eta_{j \leq j'}(t) \eta_{i' \leq i}(s')=1$.
Now
$\eps \eta_{j \leq j'}(t) \eta_{i' \leq i}(s') =
\eps \eta_{j \leq j'}(t) \eta_{j\leq j'}(t)=\eps$,
so we need $\eps = 1$.
\end{proof}

\begin{Lemma} \label{L:rel1}
\begin{align*}
   \left[ f_{i,s}^{j,t}, f_{k,p}^{l, q} \right] &=
   \delta_{i,l} \delta_{s,q} f_{k,p}^{j,t} -
   \delta_{k,j} \delta_{p,t} f_{i,s}^{l,q} \\
   &\quad - \delta_{i,k'} \delta_{s, p'} \eps \eta_{l \leq l'}(q) \eta_{i \leq i'}(s) f_{l',q'}^{j,t} \\
   &\quad + \delta_{l', j} \delta_{q', t} \eps \eta_{l \leq l'}(q) \eta_{k' \leq k}(p') f_{i,s}^{k',p'}.
\end{align*}
\end{Lemma}
\begin{proof}
 This is straight forward calculation, however if will be useful to note that
 $\eta_{i \leq i'}(s) \eta_{i' \leq i}(s') = \eps$ for all $i$.
\end{proof}

\begin{Lemma} \label{L:basis}
If $\eps = -1$, then
  \begin{align*}
&\{ f_{i,s}^{j,t} \mid 1 \leq i,j \leq n, j \notin \{i, i'\}, i< j, 0 \leq s < \lambda_i, 0 \leq t < \lambda_j \} \\
& \quad \cup \{ f_{i,s}^{j,t} \mid 1 \leq i \leq  n, j \in \{i, i'\}, s+t\leq \lambda_i-1 , 0 \leq s,t < \lambda_i\}
  \end{align*}
  is a basis of $\g$.

  If $\eps = 1$, then
  \begin{align*}
&\{ f_{i,s}^{j,t} \mid 1 \leq i,j \leq n, j \notin \{i, i'\}, i< j, 0 \leq s < \lambda_i, 0 \leq t < \lambda_j \} \\
& \quad \cup \{ f_{i,s}^{j,t} \mid 1 \leq i \leq  n, j \in \{i, i'\},
s+t< \lambda_i-1, 0 \leq s,t < \lambda_i\}
  \end{align*}
  is a basis of $\g$.
\end{Lemma}
\begin{proof}
 This follows from Lemma \ref{L:f0} and the fact that
 $f_{i,s}^{j,t} = \pm f_{j', t'}^{i', s'}$ for all $i,j,s,t$.
\end{proof}

\subsection{The Lax type operator using this realization}
In terms of our basis for $\gl(V)$, we have that
\[
 f = \sum_{i=1}^n \sum_{s=0}^{\lambda_i-2} e_{i,s}^{i, s+1}
\]
To define $U(\g,f)$, we define the symmetric nondegenerate equivariant form on $\g$ via
\[
  \langle X, Y \rangle  = \tfrac{1}{2} \operatorname{trace}(X, Y).
\]

Recall that $\bar 1 = 1+J \in U(\g)/J$ (see \eqref{L:walgdef}).
\begin{Lemma} \label{L:pr}
Let $f_{j, t}^{k, q} \in \g_{\geq 1}$.  Then
$f_{j,t}^{k,q} \bar 1 = \delta_{j,k} \delta_{t,q+1} \bar 1$.
\end{Lemma}
\begin{proof}
From the definition of $J$ we have that
\begin{align*}
f_{j,t}^{k,q} \bar 1 &= \langle f, f_{j,t}^{k,q} \rangle \bar 1 \\
&= \tfrac{1}{2} \operatorname{trace} \left( \left(
 \sum_{i=1}^n \sum_{s=0}^{\lambda_i-2} e_{i,s}^{i, s+1} \right)
 \left(e_{j,t}^{k,q}
 - \eps \eta_{k \leq k'}(q) \eta_{j' \leq j}(t')e_{k',q'}^{j',t'} \right) \right) \bar 1\\
 &=
 \tfrac{1}{2} \operatorname{trace} \left(e_{j, t}^{k, q+1}
 - \eps \eta_{k \leq k'}(q) \eta_{j' \leq j}(t')e_{k',q'}^{j',t'+1}
 \right) \bar 1 \\
 &= \delta_{j,k} \delta_{t,q+1} \bar 1.
\end{align*}
\end{proof}

Now $\{f^s w_i \mid 1 \leq i \leq n, 0 \leq s \leq \lambda_i-1\}$ is a basis of $V$ consistent with the setup from $\S$\ref{S:Lz}.
Recall that in $\S$\ref{S:Lz}
we created a pyramid from $(\lambda_1, \dots, \lambda_n)$, the Jordan type of $f$ with blocks labeled with $1, \dots, N$ where $N = \dim V$,
and that $V$ has a basis $\{e_a \mid a = 1, \dots, N$.
Furthermore for this basis $F e_a = e_b$ where
$\row(a) = \row(b)$ and $\col(b) = 1 + \col(a)$ if $\col(a) < \lambda_{\row(a)}$, and $F e_a = 0$ if $\col(a) = \lambda_{\row(a)}$.
Note
the leftmost block of the pyramid in row $i$ lies in column $(\lambda_1 - \lambda_i)/2+1$.
Thus we can identify the basis element $f^s w_i$ with $e_a$
where $\row(a) = i$ and $\col(a) = (\lambda_1-\lambda_i)/2 + 1 + s$.
We will also use the notation $\col(i,s) = (\lambda_1-\lambda_i)/2 + 1 + s$, the column of the basis element corresponding to $f^s w_i$.

So for $a,b = 1, \dots, N$ we identify
\begin{equation} \label{EQ:conv}
 e_{a,b} = e_{i,s}^{j,t}
\end{equation}
where
\begin{itemize}
 \item $i = \row(b)$,
 \item $s = \col(b) + \tfrac{\lambda_i-\lambda_1}{2}-1$,
 \item $j = \row(a)$,
 \item $t = \col(a) + \tfrac{\lambda_j-\lambda_1}{2}-1$.
\end{itemize}

If
$e_{a,b} = e_{i,s}^{j,t}$,
we define $a'$ and $b'$ via $e_{b',a'} = e_{j', t'}^{i', s'}$.
More explicitly we have
\begin{itemize}
 \item $\row(b') = i'$,
 \item $\col(b') = \lambda_1+1-\col(b)$,
 \item $\row(a') = j'$,
 \item $\col(a') = \lambda_1+1-\col(a)$,
\end{itemize}
Now we define
\[
 f_{a,b} = e_{a,b} - \eps \eta_{j \leq j'}(t) \eta_{i' \leq i}(s') e_{b',a'},
\]
so that $f_{a,b} = f_{i,s}^{j,t}$.

Now Lemma \ref{L:rel1} translates to the following:
\begin{Lemma} \label{L:rel2}
   Let $a,b,c,d = 1, \dots, N$.  Then
   \begin{align*}
   [f_{a,b},f_{c,d}] & = \delta_{b,c} f_{a,d} - \delta_{a,d} f_{c,b} \\
   & \quad -\delta_{b,d'} \eps \eta_{\row(c) \leq \row(c)'}(q) \eta_{\row(b) \leq \row(b)'}(s) f_{a,c'}  \\
   & \quad
   +\delta_{a,c'} \eps \eta_{\row(a)' \leq \row(a)}(t') \eta_{\row(d)' \leq \row(d)}(p') f_{d',b}
   \end{align*}
   where
\begin{itemize}
 \item $s = \col(b) + \tfrac{\lambda_{\row(b)}-\lambda_1}{2}-1$,
 \item $t = \col(a) + \tfrac{\lambda_{\row(a)}-\lambda_1}{2}-1$,
 \item $p = \col(d) + \tfrac{\lambda_{\row(d)}-\lambda_1}{2}-1$,
 \item $q = \col(c) + \tfrac{\lambda_{\row(c)}-\lambda_1}{2}-1$.
 \end{itemize}
\end{Lemma}

We also have that Lemma \ref{L:basis} is equivalent to:
\begin{Lemma} \label{L:basis2}
   If $\eps = -1$, then
   \begin{align*}
     &\{ f_{a,b} \mid  a,b =1, \dots, N,
        \row(a) \notin \{ \row(b), \row(b)' \},
         \row(b) < \row(a)\} \\
         &\quad \cup
       \{ f_{a,b} \mid  a,b =1, \dots, N,
        \row(a) \in \{ \row(b), \row(b)' \}, \\
        & \qquad \qquad \quad
         \col(a) +\col(b) \geq \lambda_1+1 \}
    \end{align*}
    is a basis of $\g$.

   If $\eps = 1$, then
   \begin{align*}
     &\{ f_{a,b} \mid  a,b =1, \dots, N,
        \row(a) \notin \{ \row(b), \row(b)' \},
         \row(b) < \row(a)\} \\
         &\quad \cup
       \{ f_{a,b} \mid  a,b =1, \dots, N,
        \row(a) \in \{ \row(b), \row(b)' \}, \\
        & \qquad \qquad \quad
         \col(a) +\col(b) > \lambda_1+1 \}
    \end{align*}
    is a basis of $\g$.

\end{Lemma}

Furthermore Lemma \ref{L:pr} translates to:
\begin{Lemma} \label{L:pr2}
  Let $f_{a,b} \in \g_{\geq 1}$.  Then
  $f_{a,b} \bar 1$ is $\bar 1$ if $\row(a)$ = $\row(b)$
  and $\col(b) = \col(a)+1$, and is 0 otherwise.
\end{Lemma}

\subsection{Adapting the formulas for $L_k(z)$ and $L_{k;R}(z)$  to the basis from Lemma \ref{L:basis2}}
Next we need to express the matrix $\bar Y$ from \eqref{EQ:Yz1} in
terms of our basis from Lemma \ref{L:basis2}.
Note that
our assumption
that all the parts of the Jordan type of $f$ have the same parity makes it so there are no non-zero half-integer weight elements of $\g$.
So the basis elements from \eqref{EQ:Yz1} for which $\col(a) \geq \col(b)$ form a basis for $\g_{\leq 0}$.

Note that the dual basis element of
$f_{i,s}^{j,t}$ is $f_{j,t}^{i,s}$ if $j \neq i'$ or $j= i'$
and $t \neq s'$,
and in the case $\eps = -1$ the dual basis element to
$f_{i,s}^{i',s'}$ is $e_{i,s}^{i', s'}$.
So the dual basis element to
$f_{a,b}$ is $f_{b,a}$ if $\row(a) \neq \row(b)'$ or $\row(a) = \row(b')$
and $\col(a) \neq \lambda_1 + 1 - \col(b)$,
and in the case $\eps = -1$ the dual basis element to $f_{a,b}$ where $\row(a) = \row(b)'$ and $\col(a)  = \lambda_1 + 1 - \col(b)$ is $e_{b,a}$.

So if $u_i = f_{a,b}$, then
\[
   u_i \otimes U^i =
   \begin{cases}
      f_{a,b} \otimes E_{b,a} & \text { if } \eps = -1,
                     \row(a) = \row(b)', \\
      & \quad \text{ and }
       \col(a) + \col(b) = \lambda_1+1; \\
       f_{a,b} \otimes F_{b,a} & \text{otherwise}.
   \end{cases}
\]
Furthermore in the latter case if
 $s = \col(b) + \tfrac{\lambda_{\row(b)}-\lambda_1}{2}-1$
 and
  $t = \col(a) + \tfrac{\lambda_{\row(a)}-\lambda_1}{2}-1$,
  then
  \begin{align*}
  f_{a,b} \otimes F_{b,a} &=
  f_{a,b} \otimes (E_{b,a} - \eps \eta_{\row(a) \leq \row(a)'}(t)
     \eta_{\row(b)' \leq \row(b)}(s') E_{a',b'}) \\
     &= f_{a,b} \otimes E_{b,a} + f_{b',a'} \otimes E_{a',b'}.
     \end{align*}

So whenever $\col(a) \leq \col(b)$, then $\bar Y_{a,b} = f_{b,a} z^{\col(a)-\col(b)-1}$.
Thus
we can replace each $\bar Y_{d_i, a_i}$ with $f_{a_i, d_i} z^{\col(d_i)-\col(a_i)-1}$
in Theorems \ref{T:LzRexplicit} and \ref{T:Lzexplicit}.

Recall that if $k$ is a non-negative integer, then
\[
 \g^k = \left\{ \Pi_{[-k,k]} X \Psi_{[-k,k]} \mid X \in \g \right\}.
\]
Also recall the definitions of $s_k$ and $e_k$ from \eqref{EQ:sk}.

\begin{Lemma} \label{L:Pik}
Let $k$ be any non-negative integer.  Then
  \[
\Pi_{[-k,k]} f_{i,s}^{j,t} \Psi_{[-k,k]} \neq 0
  \]
if and only if
  \[
  s_k \leq \col(i,s),\col(j,t) \leq e_k,
  \]
  in which case
  $\Pi_{[-k,k]} f_{i,s}^{j,t} \Psi_{[-k,k]} = f_{i,s}^{j,t}$.

  Equivalently,
  \[
\Pi_{[-k,k]} f_{a,b} \Psi_{[-k,k]} \neq 0
  \]
if and only if
  \[
     s_k \leq \col(a),\col(b) \leq e_k,
  \]
  in which case
  $\Pi_{[-k,k]} f_{a,b} \Psi_{[-k,k]} = f_{a,b}$.

\end{Lemma}
\begin{proof}
   A quick calculation shows that
   \[
  s_k \leq \col(i,s) \leq e_k
  \]
  if and only if
  \[
  s_k \leq \col(i',s') \leq e_k.
  \]
  So we just need to prove that
  $\Pi_{[-k,k]} e_{i,s}^{j,t} \Psi_{[-k,k]} \neq 0$ if and only if
  \[
  s_k \leq \col(i,s),\col(j,t) \leq e_k.
  \]
  To prove this, note that
  in general the $x$-weight for the basis vector $e_a \in V$ is
  $(\lambda_1 -1)/2 -\col(a)+1$.
  So
  $\Pi_{[-k,k]} e_{i,s}^{j,t} \Psi_{[-k,k]} \neq 0$ if and only if
  \[
  -k \leq (\lambda_1 -1)/2 -\col(i,s)+1 \leq k
  \]
  and
  \[
  -k \leq (\lambda_1 -1)/2 -\col(j,t)+1 \leq k,
  \]
  which is equivalent to
  \[
  s_k = (\lambda_1+1)/2-k \leq \col(i,s),\col(j,t) \leq (\lambda_1 + 1)/2+k = e_k.
  \]
  Is is also now clear that when these conditions hold
  $\Pi_{[-k,k]} e_{i,s}^{j,t} \Psi_{[-k,k]} = e_{i,s}^{j,t}$.
\end{proof}

\subsection{Proofs of Theorems \ref{T:Lk} and \ref{T:LRk}}
First we restate Theorem \ref{T:Lk} in terms of
of the coordinates in this section.

\begin{Theorem} \label{T:Lkv2}
 Let $k$ be the weight of a highest weight vector in $V$ for the $\mathfrak{sl}_2$-triple $(f,2x,e)$.
 So $\lambda_i = 2k+1$ for some $i$.
 Let $c,d \in \{1, \dots, N\}$ such that
 $\lambda_{\row(c)} = \lambda_{\row(d)} = \lambda_i, \col(c) = s_k$, and $\col(d) = e_k$.
 Then all the coefficients of $L_k(z)_{c,d}$ are in $U(\g,f)$.
\end{Theorem}

Next we restate Theorem \ref{T:LRk} in terms of the coordinates in this section.
\begin{Theorem} \label{T:LRkv2}
 Let $k$ be the weight of a highest weight vector in $V$ for the $\mathfrak{sl}_2$-triple $(f,2x,e)$, and let $l$
 be the largest weight of a highest weight vector in $V$
 which is less than $k$.  So there exists $i,j$ such that
 $\lambda_i = 2k+1$ and $\lambda_j = 2l+1$.
 Let $c,d \in \{1, \dots, N\}$ such that
 $\lambda_{\row(c)} = \lambda_j$, $\lambda_{\row(d)} = \lambda_i, \col(c) = e_l$, and $\col(d) = e_k$.
 Let $q \in \Z$ such that $q > k-l$.
 Then the coefficient of $z^{-q}$ in $L_{k;R}(z)_{c,d}$ is in
 $U(\g,f)$.
\end{Theorem}

\begin{proof}
 We will prove Theorem \ref{T:Lkv2}, however essentially the same argument works to prove Theorem \ref{T:LRkv2}.
 Note that $\g_{>1/2}$ is generated $\g_1$ (since $\g_{1/2} = 0$),
 and that $\g_1$ is spanned by $\{f_{a,b} \mid \col(b) = \col(a)+1, \col(a) \geq (\lambda_1+1)/2 \}$.
 Recall the definition of $f_k$ from \eqref{EQ:fk}.
 So if $(\mu_1 \geq \mu_2 \geq \dots \geq \mu_n)$ is the Jordan type of $f_k$, then $\lambda_i = \mu_1$

 We need to prove that
 $[ f_{a,b}, L_k(z) ]_{c,d} \bar 1 = 0$
 for all $a,b= 1, \dots, N$
 such that
 $\col(b) = \col(a)+1$.
 If $\col(a) > e_k$, then $f_{a,b}$ commutes with every element of $\g^k$, so it commutes with $L_k(z)$.
 If $s_k \leq \col(a),\col(b) \leq e_k$, then
 $[ f_{a,b}, L_k(z) ]_{c,d} \bar 1 = 0$
 by \cite[Theorem 4.9]{SKV} (if proving Theorem \ref{T:LRk}, then
 the coefficients of $z^{-p}$ for $p > (\lambda_i-\lambda_j)/2$ of $[ f_{a,b}, L_{k;R}(z) ]_{c,d}$ are 0  by Theorem \ref{T:main}).

 So the only case we have left to consider is when
 $\col(a) = e_k$ and $\col(b) = e_k+1$,
 which implies that
 $\lambda_{\row(b)} > \lambda_i$,

 Suppose that $f_{x,y}$ is part of a monomial $u f_{x,y} v$ which occurs as part of a
 sum of monomials in a coefficient of some power of $z^{-1}$ in $L_k(z)_{c,d}$.
 This means that $\col(y) \leq \col(x) \leq e_k $ since $f_{x,y} \in \g^k_{\leq 0}$.
 So by Lemma \ref{L:rel2} we have that
 $[f_{a,b}, f_{x,y}] =
 - \delta_{a,y} f_{x,b} \pm \delta_{a,x'} f_{y',b}$.

 If $a = y$,
 then
$\col(x) = e_k$
 since
 $\col(a) = e_k$ and
$\col(y) \leq \col(x) \leq e_k$.
Also by Lemma \ref{L:aanotappear} (or Lemma \ref{L:aanotappear2} if proving Theorem \ref{T:LRk}) we have that $\row(d) = \row(x)$, so $\row(x) \neq \row(b)$
since $\lambda_{\row(d)} = \lambda_i$ and $\lambda_{\row(b)} > \lambda_i$.
So by Lemma \ref{L:pr2}, $f_{x,b} \bar 1 = 0$.

Similarly if $a = x'$, then $\col(x) = s_k$ so $\col(y) = s_k$,
$\row(y') \neq \row(b)$, and $f_{y', b} \bar 1 = 0$.

Note that
\[
[f_{a,b}, u f_{x,y} v] \bar 1 = ([f_{a,b} u] f_{x,y} v +
u(-\delta_{a,y} f_{x,b} \pm \delta_{a,x'} f_{y',b})v
+ u f_{x,y} [f_{a,b}, v]) \bar 1.
\]
Now if $a=y$ then $f_{x,b}$ will need to be commuted to the right of $v$,
and once it is to the right of $v$ it will become $0$.
In the same vein, if $a=x'$, then once $f_{y', b}$ is commuted to the right of $v$ it will become $0$.

Since $f_{x,b}$ if $a = y$, and $f_{y',b}$ if $a = x'$ satisfy
the same hypotheses as $f_{a,b}$, any terms created by moving them
to the right of $v$ will also need to be moved to the right of $v$, after which they become 0.  Thus
$u \delta_{a,y} f_{x,b} v \bar 1 = 0$  and
$u \delta_{a,x'} f_{y', b} v \bar 1 = 0$.
This argument can now be repeated to show
that $[f_{a,b}, u] f_{x,y} v \bar 1$ and $u f_{x,y} [f_{a,b} v \bar 1$ are both $0$.
Therefore  $[ f_{a,b}, L_k(z) ]_{c,d} \bar 1= 0$

\end{proof}

\section{Proof of Theorem \ref{T:generate}} \label{S:proof}

There is a filtration on $U(\g,f)$
defined by declaring that if $X \in \g_{k}$ for $k \leq 1/2$, then $\deg(X) = -k$.
Then if $x = \sum_{i=1}^p X_{i_1} X_{i_2} \dots X_{i_k} \bar 1\in U(\g,f)$ where
$X_{i,j}  \in \g_{\leq 1/2}$ for all $i,j$, we define
\[
\deg(x) = \max_{i \in \{1, \dots, p\}} \left(\sum \deg(X_{i_1}) + \dots + \deg(X_{i_k}) \right).
\]

Under this filtration, the associated graded algebra
$\operatorname{gr}(U(\g,f)) \cong U(\g^f)$, where $\g^f$ is the centralizer of $f$ (see \cite[Theorem 3.8]{BGK}).
From \eqref{EQ:delta} we have that the degree of $f_{a,b}$ is $\col(b) - \col(a)$.

\begin{Lemma} \label{L:gr}
 Let $x$ be the coefficient of $z^{-p}$ in $L_R(z)_{c,d}$ from Theorem \ref{T:LzRexplicit}.  Then
 \[
   \operatorname{gr}(x) =
-\sum (-1)^{\col(c,a_1,d_1)+n(c)+\lambda_1} f_{a_1, d_1}
\]
where we are summing over all
$a_1, d_1 \in \{1, \dots, N\}$ such that
\begin{itemize}
 \item $\row(d_1) = \row(c)$,
 \item $\col(d_1) > 1$,
  \item $\row(a_{1}) = \row(d)$,
  \item $\col(d_1) \leq \col(a_1)+1/2$,
  \item $\col(a_1)-\col(d_1) +1 = p$.
  \item $\col(c) < \col(d_1)$ if $\lambda_{\row(c)} = \lambda_1$,
  \item $\col(c) \geq \col(d_1)$ if $\lambda_{\row(c)} < \lambda_1$,
\end{itemize}

Also let $y$ be the coefficient of $z^{-p}$ in $L(z)_{c,d}$ from Theorem \ref{T:Lzexplicit}.  Then
 \[
   \operatorname{gr}(y) =
-\sum (-1)^{\col(a_1,d_1)+\lambda_1}
f_{a_1,d_1}
\]
where we are summing over all
$a_1, d_1 \in \{1, \dots, N\}$ such that
\begin{itemize}
 \item $\row(d_1) = \row(c)$,
  \item $\row(a_{1}) = \row(d)$,
  \item $\col(d_1) \leq \col(a_1)$,
  \item $\col(a_1)-\col(d_1) +1 = p$.
\end{itemize}
\end{Lemma}
\begin{proof}
   We just have to note that terms in the sums
   from Theorems \ref{T:LzRexplicit} and \ref{T:Lzexplicit}
   have highest degree when $s=1$.
\end{proof}

Our goal is to show that these elements generate $\g^f$.
Following \cite[$\S$2]{PrT},
for $i,j = 1,\dots, n$ and $t = 0, 1, \dots, \lambda_j-1$ we define
\begin{equation} \label{EQ:zeta}
  \zeta_i^{j,t} =
  \sum_{s=0}^t
  e_{i,s}^{j, s+t'}.
  +
  \sigma(e_{i,s}^{j, s+t'})
  =
  \sum_{s=0}^t
  f_{i,s}^{j, s+t'}.
  \in \g^f.
\end{equation}

\begin{Lemma} \label{L:zeta2}
Let $k$ be a highest weight for the $\mathfrak{sl}_2$-triple $(f, 2x, e)$.
Let $c,d \in \{1, \dots, N\}$ be such that
$\lambda_{\row(c)} = \lambda_{\row(d)} = 2k+1$,
$\col(c) = s_k$, and $\col(d) = e_k$.
Let $p \in \{1, \dots, \lambda_{\row(c)}\}$.
Let $i = \row(d)$ let $j = \row(c)$, and let $t = \lambda_i-p$.
Let $y$ be the coefficient of $z^{-p}$ in $L_k(z)_{c,d}$ from Theorem \ref{T:Lzexplicit}.  Then
\[
 \operatorname{gr}(y) = (-1)^{p+\lambda_i} \zeta_{i}^{j, t}
\]
\end{Lemma}
\begin{proof}
We consider the terms
$f_{i,s}^{j,s+t'}$
in \eqref{EQ:zeta}.
Since $\lambda_j = \lambda_i$, by using the conversion formulas from \eqref{EQ:conv} we have that
$f_{i,s}^{j,s+t'} = f_{a,b}$ where $\row(a) = j$, $\col(a) = s+t'+(\lambda_1-\lambda_i)/2+1=s-t+(\lambda_1+\lambda_i)/2$, $\row(b) = i$, and
$\col(b) = s+(\lambda_1-\lambda_i)/2+1$.
So $\col(a)-\col(b) = -t-1+\lambda_i = p -1$.
Furthermore since $0 \leq s \leq t$, we have that
the minimum value for $\col(a)$ is
$-t+(\lambda_1+\lambda_i)/2 = p-\lambda_i+(\lambda_1+\lambda_i)/2 = p + (\lambda_1 - \lambda_i)/2=p+s_k-1$,
while the maximum value for $\col(a)$ is $(\lambda_1+\lambda_i)/2 = e_k$.
Since $\col(b) = \col(a)-p+1$, we have that the minimum value of
$\col(b)$ is $s_k$, while the maximum value is $e_k-p+1$.
So we are summing over all the values of $a_i, d_i$ from $\operatorname{gr}(y)$ in Lemma \ref{L:gr}.
Finally, we simply note that
$-(-1)^{\col(a_1,d_1)+\lambda_i} = -(-1)^{p-1+\lambda_i}$.

\end{proof}

A similar arguement gives the following lemma:
\begin{Lemma} \label{L:zeta1}
Let $k$ be a highest weight for the $\mathfrak{sl}_2$-triple $(f, 2x, e)$.
Let $l$ be the largest highest weight which is less than $k$.
So there exists $i,j$ such that $\lambda_i = 2l+1$
and $\lambda_j = 2k+1$.
Let $c,d \in \{1, \dots, N\}$ be such that
$\lambda_{\row(c)} = \lambda_i$, $\lambda_{\row(d)} = \lambda_j$,
$\col(c) = e_l$, and $\col(d) = e_k$.
Let $p \in \{1, \dots, \lambda_{j}\}$ such that
$p > k-l$
Let $t = (\lambda_i+\lambda_j)/2-p$.
Let $x$ be the coefficient of $z^{-p}$ in $L_{R;k}(z)_{c,d}$ from Theorem \ref{T:LzRexplicit} .
Then.
\[
 \operatorname{gr}(x) = (-1)^{e_l + 1 +p+\lambda_j} \zeta_{i}^{j, t}
\]
\end{Lemma}

\begin{Theorem} \label{T:inUgf}
   Let $i,j = 1, \dots, n$ be such that either $\lambda_j = \lambda_i$
   or $\lambda_i$ is the largest part of the Jordan type of $f$ which is less than $\lambda_j$.
   Let $t \in \{0, \dots, \lambda_i -1 \}.$
   Then $\zeta_i^{j,t} \in \operatorname{gr}(U(\g,f))$.
\end{Theorem}

\begin{proof}
In the case that $\lambda_i < \lambda_j$, if $p,l,k$ are as in Lemma \ref{L:zeta1}, note that
$p > k-l = (\lambda_j - \lambda_i)/2$,
so $t = (\lambda_j+\lambda_i)/2-p < \lambda_i$.
 Now the theorem follows from Theorems \ref{T:Lkv2} and \ref{T:LRkv2} as well as
 Lemmas \ref{L:zeta2} and \ref{L:zeta1}.
\end{proof}

Now the only thing needed to prove Theorem \ref{T:generate} is to show that
the elements from Theorem \ref{T:inUgf} generate $\g^f$.

The following lemma gives a basis for $\g^f$ and is proved in \cite[$\S$2]{PrT}:
\begin{Lemma} \label{L:basis3}
    The following a basis of $\g^f$:
    \begin{align*}
      &\{ \zeta_i^{i,t} \mid i < i', 0 \leq t < \lambda_i \} \cup \\
      & \quad
      \{ \zeta_i^{i,t} \mid i = i', 0 \leq t < \lambda_i, \lambda_i-t \text{ is even } \} \cup \\
      & \quad
      \{ \zeta_i^{i',t} \mid i \neq i', 0 \leq t < \lambda_i, \lambda_i-t \text{ is odd } \} \cup \\
      & \quad
      \{ \zeta_i^{j,t} \mid i < j \neq i', 0 \leq t < \lambda_i\}.
    \end{align*}
\end{Lemma}
Note that if $i>j$ and $0 \leq t< \lambda_j$, $\zeta_i^{j,t} = \pm \zeta_{j'}^{i',t}$.
So all of elements from Theorem \ref{T:inUgf} are in this basis (up to a sign) except the $\zeta_i^{j,t}$'s where there exists $k$ such that
$\lambda_i < \lambda_k < \lambda_j$.
However, the following lemma shows that these elements are generated by
the $\zeta_i^{j,t}$'s where there is no such $k$, thus completing the proof of
Theorem \ref{T:generate}.
\begin{Lemma}
Let $i, j, k \in \{1, \dots, n\}$ be such that
$\lambda_i < \lambda_k < \lambda_j$, and let $0 \leq s < \lambda_i$ and $0 \leq t < \lambda_k$.  Then
\[
  \left[ \zeta_i^{k, s}, \zeta_k^{j, t} \right] =
  - \zeta_i^{j, s+t - (\lambda_k-1)}
\]
\end{Lemma}
\begin{proof}
 This is immediate from \cite[Lemma 5]{PrT}.
\end{proof}
Since in this lemma we can take $t = \lambda_k-1$ and $s$ to be arbitrary,
we therefore have that every element in the basis from Lemma \ref{L:basis3} is in $\operatorname{gr}(U(\g,f))$.
Thus we have proven Theorem \ref{T:generate}.

\end{document}